\documentclass[a4paper,11pt,reqno]{amsart}
\usepackage{amsthm}
\usepackage{amsmath}
\usepackage{enumerate}
\usepackage{amsfonts}
\usepackage{amssymb}
\usepackage{fullpage}
\usepackage{amsmath,amscd}
\usepackage{stmaryrd}
\usepackage{graphicx}
\usepackage{tikz-cd}
\usepackage{array}
\usepackage{amsmath,amssymb,amsfonts,dsfont}
\usepackage[utf8]{inputenc}
\usepackage[english]{babel}
\usepackage[T1]{fontenc}
\usepackage{graphicx}
\usepackage{float}
\usepackage{pdfpages}
\usepackage[a4paper, margin = 3cm, bottom = 3cm]{geometry}
\usepackage{ifpdf}
\usepackage{marginnote}
\usepackage{hyperref}	

\usetikzlibrary{calc}
\usetikzlibrary{decorations.pathreplacing,decorations.markings,decorations.pathmorphing}
\usetikzlibrary{positioning,arrows,patterns}
\usetikzlibrary{cd}
\usetikzlibrary{intersections}
\usetikzlibrary{arrows}

\hypersetup{pdfborder=0 0 0,
	    colorlinks=true,
	    citecolor=black,
	    linkcolor=blue,
	    urlcolor=red,
	    pdfauthor={Guillaume Tahar}
	   }
\newtheorem{thm}{Theorem}[section]
\newtheorem{cor}[thm]{Corollary}
\newtheorem{prop}[thm]{Proposition}
\newtheorem{lem}[thm]{Lemma}
\newtheorem{conj}[thm]{Conjecture}
\theoremstyle{definition}
\newtheorem{defn}[thm]{Definition}
\theoremstyle{remark}

\theoremstyle{definition}

\theoremstyle{definition}

\theoremstyle{definition}

\numberwithin{equation}{section}
\title{Horizon saddle connections and Morse-Smale dynamics of dilation surfaces}

\author{Guillaume Tahar}
\address[Guillaume Tahar]{Beijing Institute of Mathematical Sciences and Applications, Huairou District, Beijing, China}
\email{guillaume.tahar@bimsa.cn}

\date{February 9, 2023}

\begin{document}

\begin{abstract}
Dilation surfaces are generalizations of translation surfaces where the transition maps of the atlas are translations and homotheties with a positive ratio. In contrast with translation surfaces, the directional flow on dilation surfaces may contain trajectories accumulating on a limit cycle. Such a limit cycle is called hyperbolic because it induces a nontrivial homothety. It has been conjectured that a dilation surface with no actual hyperbolic closed geodesic is in fact a translation surface. Assuming that a dilation surface contains a horizon saddle connection, we prove that the directions of its hyperbolic closed geodesics form a dense subset of $\mathbb{S}^{1}$. We also prove that a dilation surface satisfies the latter property if and only if its directional flow is Morse-Smale in an open dense subset of $\mathbb{S}^{1}$.\newline
\end{abstract}
\maketitle
\setcounter{tocdepth}{1}
\tableofcontents

\section{Introduction}

Let us consider a compact Riemann surface $X$ with a finite set $\Sigma$ of marked points. A \textit{dilation structure} is an atlas of charts on $X \setminus \Sigma$ modelled on the complex plane $\mathbb{C}$ with transition maps of the form $z \mapsto az+b$ with $a \in \mathbb{R}_{+}^{\ast}$ and $b \in \mathbb{C}$.\newline
In a dilation surface, one can define what a straight line is, as well as a slope. Therefore one can define a directional flow on the surface (leaves or trajectories are locally conjugated to straight lines in each chart). Just like translation surfaces can be understood as suspensions of interval exchange maps, dilation surfaces can be thought of as suspensions of affine interval exchange maps. However, their holonomy does not preserve any metric.\newline
We say that a dilation surface is \textit{strict} if it is not a translation surface. Moreover, a closed geodesic in a dilation surface is said to be \textit{hyperbolic} if its monodromy representation has a nontrivial dilation ratio (the coefficient of the affine transition map). We state three conjectures of increasing strengths.

\begin{conj}\label{conj:one}
Every strict dilation surface contains a hyperbolic closed geodesic.
\end{conj}

Conjecture \ref{conj:one} has been proposed in \cite{DFG}.

\begin{conj}\label{conj:two}
Let $X$ be a strict dilation surface. The set of directions of hyperbolic closed geodesics of $X$ forms a dense subset of $\mathbb{S}^{1}$.
\end{conj}

A weak version of Conjecture~\ref{conj:two} has been proved in \cite{BGT}. It asserts that directions of closed geodesics are dense in $\mathbb{S}^{1}$. Unfortunately no information on whether a significant part of these geodesics are hyperbolic is provided.

\begin{conj}\label{conj:three}
Let $X$ be a strict dilation surface. There is an open dense subset $U$ of $\mathbb{S}^{1}$ such that the directional flow of $X$ along any direction of $U$ is Morse-Smale.
\end{conj}

The Morse-Smale dynamics means that every trajectory either hits a singularity or accumulates on a hyperbolic closed geodesic. These kind of dynamics never happen in translation surfaces where generic trajectories are either periodic or minimal in some domain (see for example Proposition 5.5 in \cite{Ta1} for a recent reference).\newline

An even stronger (and more interesting) form of Conjecture~\ref{conj:three} is that the open dense set has full measure.\newline

The first main result of this paper is that Conjectures~\ref{conj:two}~and~\ref{conj:three} are in fact equivalent to each other.

\begin{thm}\label{thm:equi}
A dilation surface satisfying Conjecture \ref{conj:two} also satisfies Conjecture \ref{conj:three}.
\end{thm}

The second main theorem of the paper is that Conjecture~\ref{conj:three} holds for a wide class of dilation surfaces. The notion of \textit{horizon saddle connection} has been introduced in \cite{Ta}. They are saddle connections such that the number of intersection points with any trajectory is globally bounded (see Subsection~\ref{sec:Horizon} for details). There is no such saddle connection in translation surfaces.\newline
Moreover, supposing such a saddle connection can be found in a dilation surface implies strong constraints on its Veech group and an almost trivial action of $GL^{+}(2,\mathbb{R})$ (see Theorems 1.3 to 1.5 in \cite{Ta}).\newline

Dilation surfaces with horizon saddle connections include several classes already studied:
\begin{itemize}
\item Dilation surfaces with hyperbolic cylinders of angle at least $\pi$;
\item A subclass of quasi-Hopf surfaces (see Proposition 3.7 in \cite{Ta});
\item Two-chamber surfaces (class described in Subsection 3.1 of \cite{DFG} as forming an exotic connected component of the moduli space $D_{2,1}$ of dilation surfaces of genus two with one singularity, see Figure 1 for an example).\newline
\end{itemize}

\begin{figure}
\includegraphics[scale=0.7]{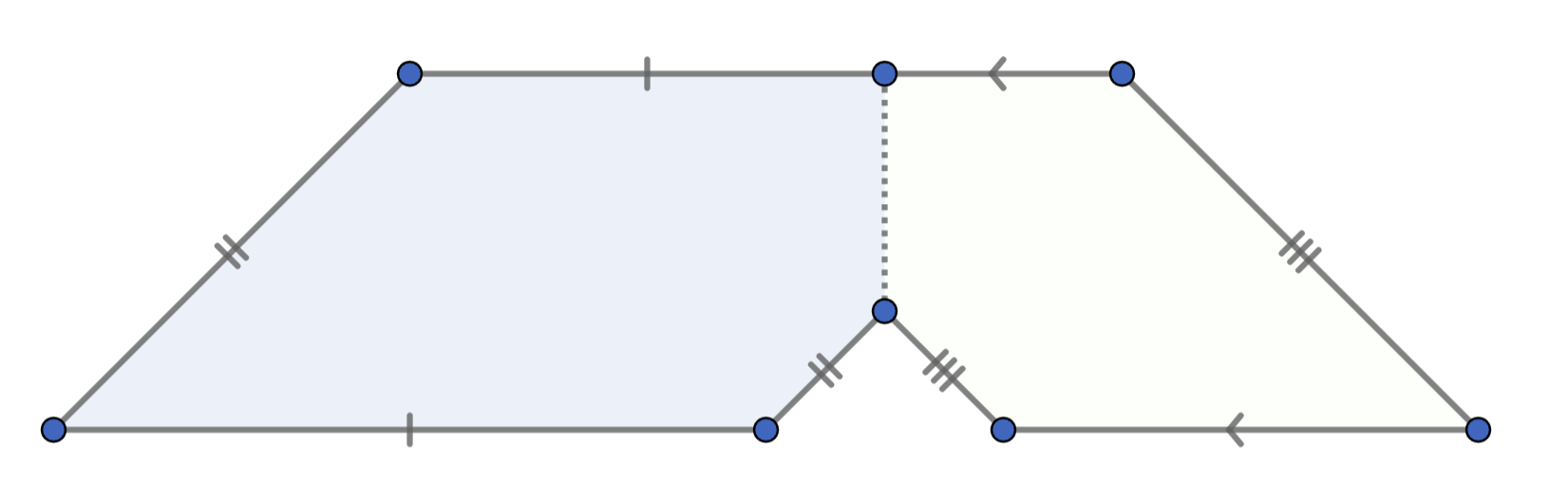}
\caption{A two-chamber surface. The dashed saddle connection between the two chambers is a horizon saddle connection because its complement in the surface is not connected. No trajectory can intersect it twice.} 
\end{figure}

\begin{thm}\label{thm:horizon}
Dilation surfaces containing at least one horizon saddle connection satisfy Conjecture~\ref{conj:three}.
\end{thm}

The two results we focus on are about dilation surfaces without boundary. However, proving these two results involves using surfaces with boundary. That is why most intermediate results in this paper will apply to the case of surfaces with boundary.\newline

The organization of the paper is the following:
\begin{itemize}
\item In Section~\ref{sec:Gene}, we give background on dilation structures, their trajectories and horizon saddle connections.
\item In Section~\ref{sec:HC}, we define hyperbolic cylinders and prove preliminary results that involve such cylinders.
\item Exceptional pencils are introduced in Section~\ref{sec:EPD}. We then prove a general result of existence for hyperbolic closed geodesics in dilation surfaces with boundary.
\item In Section~\ref{sec:MSD}, we prove Theorems \ref{thm:equi} and \ref{thm:horizon}.\newline
\end{itemize}

\section{Generalities on dilation surfaces}\label{sec:Gene}

\begin{defn}
A dilation surface is a compact topological surface $X$ (possibly with nonempty boundary) with a finite set $\Sigma \in X$ of
singularities and an atlas of charts on $X \setminus \Sigma$ with values in $\mathbb{C}$ and such that:
\begin{itemize}
\item transition maps are of the form $z \mapsto az+b$ with $a \in \mathbb{R}_{+}^{\ast}$ and $b \in \mathbb{C}$;
\item the geometric structure extends to every element of $\Sigma$ with a local model characterized by its topological index and its dilation ratio (see Subsection 2.2);
\item the boundary is a finite union of saddle connections (see Subsection 2.3).
\item there is at least one element of $\Sigma$ on each boundary component.
\end{itemize}
\end{defn}

\subsection{Linear holonomy}

In a dilation surface $X$, we can cover every closed path $\gamma$ of $X \setminus \Sigma$ with charts of the atlas. The transition map between the first chart and the last chart is an affine map and its linear part is well-defined. This number is a topological invariant called the \textit{dilation ratio} of the loop.\newline
Since directions are defined globally, every loop can also be given a \textit{topological index} (thereby generalizing the winding number in the case of the flat plane). The local geometry of a conical singularity is characterized by a linear holonomy $\rho(\gamma)$ and a topological index $i(\gamma)$ where $\gamma$ is a positive simple loop around the singularity. It is obtained by cutting along a ray in a flat cone of angle $i(\gamma)2\pi$ and then identifying the two sides by a homothety of ratio $\rho(\gamma)$.\newline

The local model is slightly different for singularities belonging to the boundary of the surface. They simply are flat cones with an arbitrary angle between the two boundary geodesic rays.\newline

A \textit{marked point} is a singularity of angle $2\pi$ (or of angle $\pi$ if ever it belongs to the boundary) and for which the dilation ratio is trivial.\newline

\subsection{Trajectories}

Since directions are well-defined in a dilation structure, geodesics (arcs locally conjugated to straight lines in the charts) either are simple closed geodesics or are non-closed and have no self-intersection.\newline
A \textit{saddle connection} is a geodesic segment of $X \setminus \Sigma$ with both endpoints being conical singularities.\newline
A \textit{separatrix} is a geodesic line starting from a conical singularity.

\begin{defn}
In a dilation surface $X$, a \textit{pencil} $P(x_{0},U)$ is a one-parameter family of trajectories starting from some point $x_{0} \in X$ with a direction belonging to an open interval $U \subset \mathbb{S}^{1}$.
\end{defn}

It should be noted that if $x_{0}$ is a regular point, then pencil $P(x_{0},U)$ is completely characterized by its starting point and its interval of directions. However, if $x_{0}$ is a conical singularity, then there may be several such pencils.\newline

We state some general results showing how the dynamics of separatrices control the dynamics of the other trajectories.

\begin{prop}\label{prop:separatrix}
For any dilation surface $X$ with boundary $\partial X$ and a direction $d \in \mathbb{S}^{1}$, if every separatrix in direction $d$ or $-d$ crosses the boundary in some regular point, then every trajectory in direction $d$ or $-d$ either crosses the boundary or hits a singularity.
\end{prop}

\begin{proof}
Without loss of generality, we can assume that $d$ is vertical. Since separatrices travel between singularities and boundary components, they have to be compact segments. In a given direction, there are finitely many such separatrices (each conical singularity has a finite angle). Cutting along them, we decompose $X$ into connected components $X_{1},\dots,X_{k}$.\newline

We then consider one of these components $X_{i}$. It is a dilation surface with boundary $\partial X_{i}$. Let $x$ be a singularity of $X_{i}$. There are two cases:
\begin{itemize}
    \item[(i)] $x$ is the intersection of a vertical separatrix of $X$ with a regular point of $\partial X$;
    \item[(ii)] $x$ is a conical singularity of $X$.
\end{itemize}
In case (i), the conical angle of $X_{i}$ at $x$ is strictly smaller than $\pi$.\newline
In case (ii), if the conical angle of $X_{i}$ at $x$ is strictly larger than $\pi$, then the interior of $X_{i}$ contains a vertical separatrix (this contradicts the definition of $X_{i}$).\newline
Consequently, at each singularity of $X_{i}$ the magnitude of the angular sector is at most $\pi$. In particular, every singularity of $X_{i}$ belongs to $\partial X_{i}$ (conical angles of interior singularities are integer multiples of $2\pi$).\newline
Up to adding an arbitrary number of marked points, we can decompose $X_{i}$ into flat triangles and apply the Gauss-Bonnet formula. The total angle defect of $X_{i}$ is nonnegative so that $X_{i}$ is either a polygon (contractible domain with a boundary formed by saddle connections) or a topological cylinder such that the angle at each singularity of $\partial X_{i}$ is equal to $\pi$.\newline
In the latter case, this implies that singularities of $\partial X_{i}$ are singularities of $X_{i}$ and therefore $\partial X_{i}$ is formed by saddle connections of $X$. These saddle connections are vertical (otherwise the interior of $X_{i}$ would contain some vertical separatrix). The hypothesis that every vertical separatrix crosses the boundary in a regular point implies that there is no vertical saddle connection. We get a contradiction and thus $X_{i}$ is a polygon.\newline 
Since polygons are simply connected, their dilation structure reduces to a translation structure. In a translation surface of finite area, a trajectory can hit a singularity, cross the boundary, be minimal in some domain or be periodic (see Proposition 5.5 in \cite{Ta1} for the classification of invariant components). Minimal trajectories only appear in surfaces of genus at least one and every loop in a polygon is contractible. Therefore, in each polygon, every vertical trajectory crosses $\partial X_{i}$. Boundary segments of $\partial X_{i}$ are either subsets of saddle connections of $X$ or vertical separatrices of $X$. However, a vertical trajectory cannot intersect a vertical separatrix. Consequently, every vertical trajectory in $X_{i}$ crosses a boundary saddle connection of $X$.\newline
\end{proof}

\subsection{Horizon saddle connections}\label{sec:Horizon}

The following notion has been introduced in \cite{Ta} as a geometric feature of some strict dilation surfaces.

\begin{defn}
In a dilation surface $X$, a \textit{horizon saddle connection} is a saddle connection $L$ for which there exists a number $k$ such that no trajectory of $X$ crosses $L$ more than $k$ times. Also, we will say that a closed geodesic crossing $L$ crosses it infinitely many times.
\end{defn}

In a translation surface without boundary, generic trajectories are dense in the whole surface. Therefore, given a saddle connection we can find a trajectory crossing it infinitely many times. There is no horizon saddle connection in a translation surface.\newline

The bound on the number of intersection points of horizon saddle connections can be refined by considering trajectories traveling in a given direction.

\begin{lem}
For a \textit{horizon saddle connection} $L$ in a dilation surface $X$ and for any $d \in \mathbb{S}^{1}$, we denote by $k(d)$ the maximal number of intersection points a trajectory belonging to direction $d$ or $-d$ may have with $L$. Subset $S(r) \subset \mathbb{S}^{1}$ defined by the condition $k(d) \geq r$ is an open set.
\end{lem}

\begin{proof}
We simply have to note that for any trajectory $T$ having $k$ intersection points with $L$, slight enough deviations of trajectory $T$ provide trajectories with at least $k$ intersection points with $L$. 
\end{proof}

In order to prove Conjecture \ref{conj:three} for dilation surfaces with horizon saddle connections, we  will use the following lemma.

\begin{lem}\label{lem:horipencil}
Let us consider a (closed) dilation surface $X$ with a horizon saddle connection $L$ and an open subset $U \subset \mathbb{S}^{1}$. The surface (with boundary) $X \setminus L$ contains a pencil $P$ of trajectories that has directions contained in $U$ and such that no trajectory of $P$ crosses the boundary of $X \setminus L$.
\end{lem}

\begin{proof}
Let $k$ be the largest integer such that $U \cap S(k)$ is not empty. Integer $k$ is the maximal number of intersection points a trajectory may have with saddle connection $L$ when the direction of the trajectory belongs to $U$.\newline
Let trajectory $T$ belonging to direction $d \in \mathbb{S}^{1}$ be a trajectory with $k$ intersection points with $L$. Let $x_{1},\dots,x_{k} \in L$ be the intersection points (with the order induced by direction $d$). There is an open neighborhood $\Omega$ of $d$ such that every trajectory starting from $x_{k}$ in a direction chosen in $-\Omega$ has at least $k$ intersection points with $L$. This implies that trajectories in $P(x_{k},\Omega)$ cannot intersect $L$. Otherwise the bound property that defines $k$ would be violated.\newline
\end{proof}

\section{Hyperbolic cylinders}\label{sec:HC}

A closed geodesic is said to be \textit{hyperbolic} if its holonomy has a nontrivial dilation ratio. By convention, hyperbolic geodesics are always oriented in such a way that their holonomy is contracting. Therefore, their direction can be defined unambiguously in $\mathbb{S}^{1}$ (and not only in $\mathbb{RP}^{1}$).\newline
Following \cite{DFG}, hyperbolic closed geodesics form one-parameter families called \textit{hyperbolic cylinders}. Each of these cylinders is a portion of annulus (where the inner and outer arcs are identified) covering some angular sector and bounded by saddle connections, see Figure~2. We state a first lemma.

\begin{figure}
\includegraphics[scale=0.9]{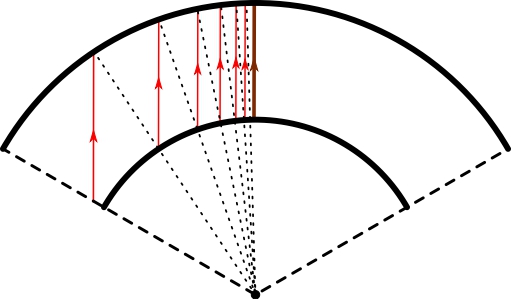}
\caption{A hyperbolic cylinder with a positive vertical trajectory accumulating on a hyperbolic closed geodesic. The inner circular arc is identified by the outer circular arc.} 
\end{figure}

\begin{lem}\label{lem:hypcyl}
If a trajectory $T$ traveling in direction $d$ enters a hyperbolic cylinder $C$ containing hyperbolic closed geodesics of direction $d$ or $-d$, then $T$ accumulates on one of these hyperbolic closed geodesics.
\end{lem}

\begin{proof}
Trajectory $T$ cannot cross a hyperbolic closed geodesic $\gamma$ belonging to the same direction (or the opposite). Moreover if $T$ enters cylinder $C$ through one boundary component, it cannot reach the other boundary component without crossing $\gamma$. Thus, $T$ never leaves the cylinder and accumulates on some limit cycle inside $C$, see Figure~2.
\end{proof}

In a given direction, there cannot be infinitely many distinct hyperbolic closed geodesics.

\begin{lem}\label{lem:finitegeo}
Let us consider a (compact) dilation surface $X$. For any direction $d \in \mathbb{S}^{1}$, $X$ contains at most finitely many hyperbolic closed geodesics pointing in direction $d$.
\end{lem}

\begin{proof}
Closed geodesics in direction $d$ are disjoint because trajectories belonging to the same direction cannot intersect each other. In a compact surface (possibly punctured at the singularities), the number of homotopically distinct disjoint loops is topologically bounded. Thus, if $X$ contains infinitely many hyperbolic closed geodesics, there have to be infinitely many hyperbolic closed geodesics in the same homotopy class $[\alpha]$.\newline
If two hyperbolic closed geodesics $\alpha$ and $\alpha'$ belong to the same direction $d$ and the same homotopy class, then they bound a topological cylinder $C$ that is automatically a hyperbolic cylinder. Indeed, $\alpha$ belongs to a $1$-parameter family of closed geodesics with the same linear holonomy. That family can be extended to $\alpha'$. Cylinder $C$ is a hyperbolic cylinder the angle of which is an integer multiple of $\pi$ (otherwise it would not contain both $\alpha$ and $\alpha'$). Consequently, if homotopy class $[\alpha]$ contains infinitely many distinct hyperbolic closed geodesics that have $d$ as direction, these closed geodesics have to belong to a unique hyperbolic cylinder of infinite angle. We assumed $X$ is compact so we get a contradiction. There are finitely many hyperbolic closed geodesics of direction $d$ in $X$.
\end{proof}

We prove another lemma that will be useful when we prove the existence of hyperbolic cylinders.

\begin{lem}\label{lem:immercone}
Let $x_{0}$ be a point in a dilation surface $X$ and $U$ be an open interval of $\mathbb{S}^{1}$. Considering pencil $P(x_{0},U)$, if no trajectory hits a singularity or crosses the boundary, then $X$ contains a hyperbolic closed geodesic in a direction contained in $U$ or $-U$.
\end{lem}

\begin{proof}
Let $T$ be a trajectory of the pencil $P(x_{0},U)$ and $y$ be one of the accumulation points of $T$. Without loss of generality, we may assume that $y$ is not a singularity of $X$. Indeed, if a conical singularity $A$ of $X$ is an accumulation point of $T$, then there is a separatrix $S$ starting from $A$ in the direction of $T$ such that every point of $S$ is also an accumulation point of $T$. In that case we can choose $y$ among regular points of $S$. Note that $y$ cannot belong to the boundary of $X$ because this would imply that some trajectory close to $T$ in $P(x_{0},U)$ crosses the boundary.\newline
Let $D$ be a small disk around $y$ for the local metric of $X$. The intersection between $D$ and $T$ is formed by infinitely many parallel segments accumulating on a subset that contains at least the diameter $[y_{0}y_{1}]$ of the disk in the direction of $T$.\newline

\begin{figure}
\includegraphics[scale=0.7]{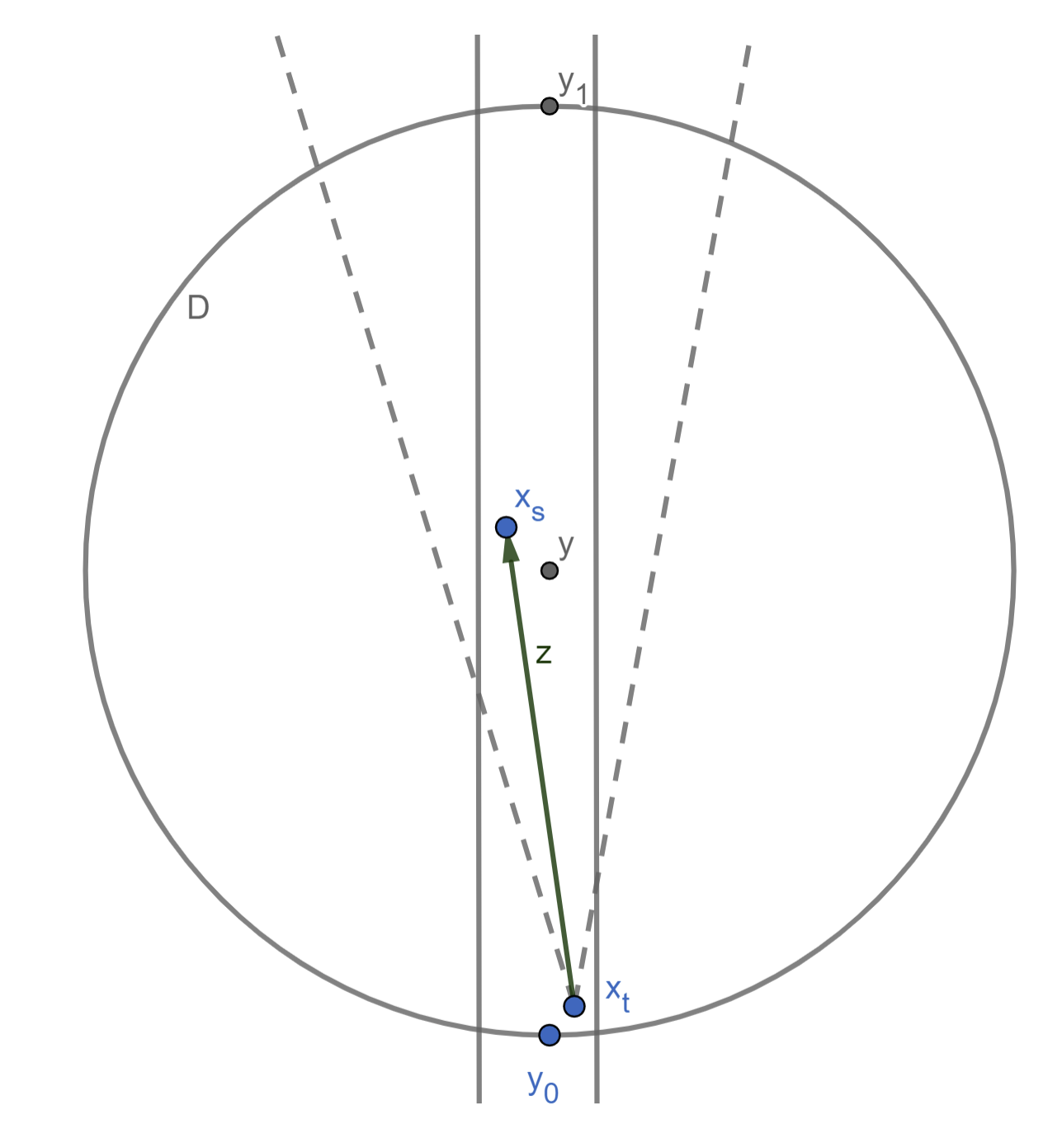}
\caption{Disk $D$ with a pencil of trajectories starting from $x_{t}$ and containing $x_{s}$.} 
\end{figure}

Let us first assume that pencil $P(x_{0},U)$ coincides with the image $\phi(\mathcal{C})$ of a infinite cone $\mathcal{C}=\left\lbrace re^{i\theta} \vert ~ r \geq 0, \theta \in U \right\rbrace$ by an affine immersion $\phi$. Trajectory $T$ is formed by points $x_{t}=\phi(te^{i\theta_{0}})$ where $t \in \mathbb{R}^{+}$.\newline
One can find a point $x_{s}$ that belongs to $D \cap P(y_{0},U)$. Since $y_{0}$ is an accumulation point of $T$, one can also find $t>s$ such that $x_{t}$ is arbitrarily close to $y_{0}$ and thus such that $x_{s} \in P(x_{t},U)$ (see Figure~3). Since $P(x_{0},U)$ is the image of the immersion of an infinite cone, we must have $P(x_{t},U) \subset P(x_{0},U)$. Consequently, the same point $x_{s}$ is the image of two distinct points that belong to cone $\mathcal{C}$. As such $x_{s}=\phi(se^{i\theta_{0}})=\phi(te^{i\theta_{0}}+z)$ where $t>s$ and $z \in \mathcal{C}$. Let $I \subset \mathcal{C}$ be the segment joining $se^{i\theta_{0}}$ and $te^{i\theta_{0}}+z$ in $\mathcal{C}$. Since $(t-s)e^{i\theta_{0}}+z \in \mathcal{C}$, the direction of segment $I$ belongs to $U \cup -U$. Besides, $\phi(se^{i\theta_{0}})=\phi(te^{i\theta_{0}}+z)$ so $\phi(I)$ is a closed geodesic $\alpha$ of $X$ with direction belonging to $U \cup -U$.\newline
Every such geodesic $\alpha$ belongs to a cylinder. If the cylinder is flat, then for any point $A \in \alpha$, some trajectory in pencil $P(A,U)$ hits a boundary singularity $B$ of the flat cylinder. Since pencil $P(A,U)$ is contained in $P(x_{0},U)$, then one trajectory of $P(x_{0},U)$ hits singularity $B$ and we get a contradiction. Therefore, $\alpha$ is a hyperbolic closed geodesic.\newline

On the other hand, if pencil $P(x_{0},U)$ is not an immersion of an infinite cone, then $P(x_{0},U)$ never the less contains the image of immersion $\phi$ of a sector $\mathcal{S}=\left\lbrace re^{i\theta} \vert ~ 0 \leq r < 1, \theta \in U \right\rbrace$ that does not extend to $\left\lbrace re^{i\theta} \vert ~ 0 \leq r \leq 1, \theta \in U \right\rbrace$. Without loss of generality, we can assume that $T$ is the image of a ray $\phi(re^{i\theta})$ such that $\phi$ does not extend to $e^{i\theta}$. Just like previously, we can consider a regular point $y$ which is an accumulation point of $T$ and a small disk $D$ centered on $y$. Following the same idea, we obtain a segment in $\mathcal{S}$ whose image is a hyperbolic closed geodesic with a direction belonging to $U \cup -U$.\newline
\end{proof}

The only obstruction for dilation surfaces to be able to decompose into flat triangles whose sides are saddle connections and vertices are singularities appears in the case of cylinders of angle at least $\pi$.

\begin{thm}\label{thm:Veech}(Veech's criterion)
For a dilation surface $X$, the three following propositions are equivalent:
\begin{itemize}
\item $X$ has a geodesic triangulation;
\item $X$ does not contain a hyperbolic cylinder of angle at least $\pi$;
\item every affine immersion of the open unit disk $\mathbb{D} \subset \mathbb{C}$ in $X$ extends continuously to its closure $\bar{\mathbb{D}}$.   
\end{itemize}
\end{thm}

The initial proof of the latter result is contained in unpublished notes (see \cite{V}). A proof of equivalence between the three propositions of Theorem \ref{thm:Veech} is given in the appendix of \cite{DFG}. The result clearly extends to dilation surfaces with boundary.\newline

We deduce from Theorem \ref{thm:Veech} an important technical result that we will use in the next Section to prove the main results of the paper.

\begin{lem}\label{lem:triangle}
Let us consider a dilation surface $X$ with no hyperbolic cylinder of angle at least $\pi$, with a singularity $A$ and a boundary saddle connection $B$. Let us also assume that there is a trajectory $\gamma$ from $A$ to a regular point $b$ of saddle connection $B$. Let $\theta$ be the (oriented) direction of $\gamma$. Then there is an embedded flat triangle $AMN$ in $X$ that satisfies the following conditions (see Figure~4):
\begin{itemize}
\item $AMN$ is bounded by three saddle connections $[AM]$, $[AN]$ and $[MN]$, where $M,N$ are singularities of $X$;
\item denoting by $]\theta_{0},\theta_{1}[$ the open angular sector of vertex $A$ inside $AMN$, trajectory $\gamma$ (with the opposite orientation) belongs to pencil $P(A,]\theta_{0},\theta_{1}[)$;
\item every trajectory in pencil $P(A,]\theta_{0},\theta_{1}[)$ crosses boundary saddle connection $B$ at some regular point.
\end{itemize}
\end{lem}

\begin{figure}
\includegraphics[scale=0.7]{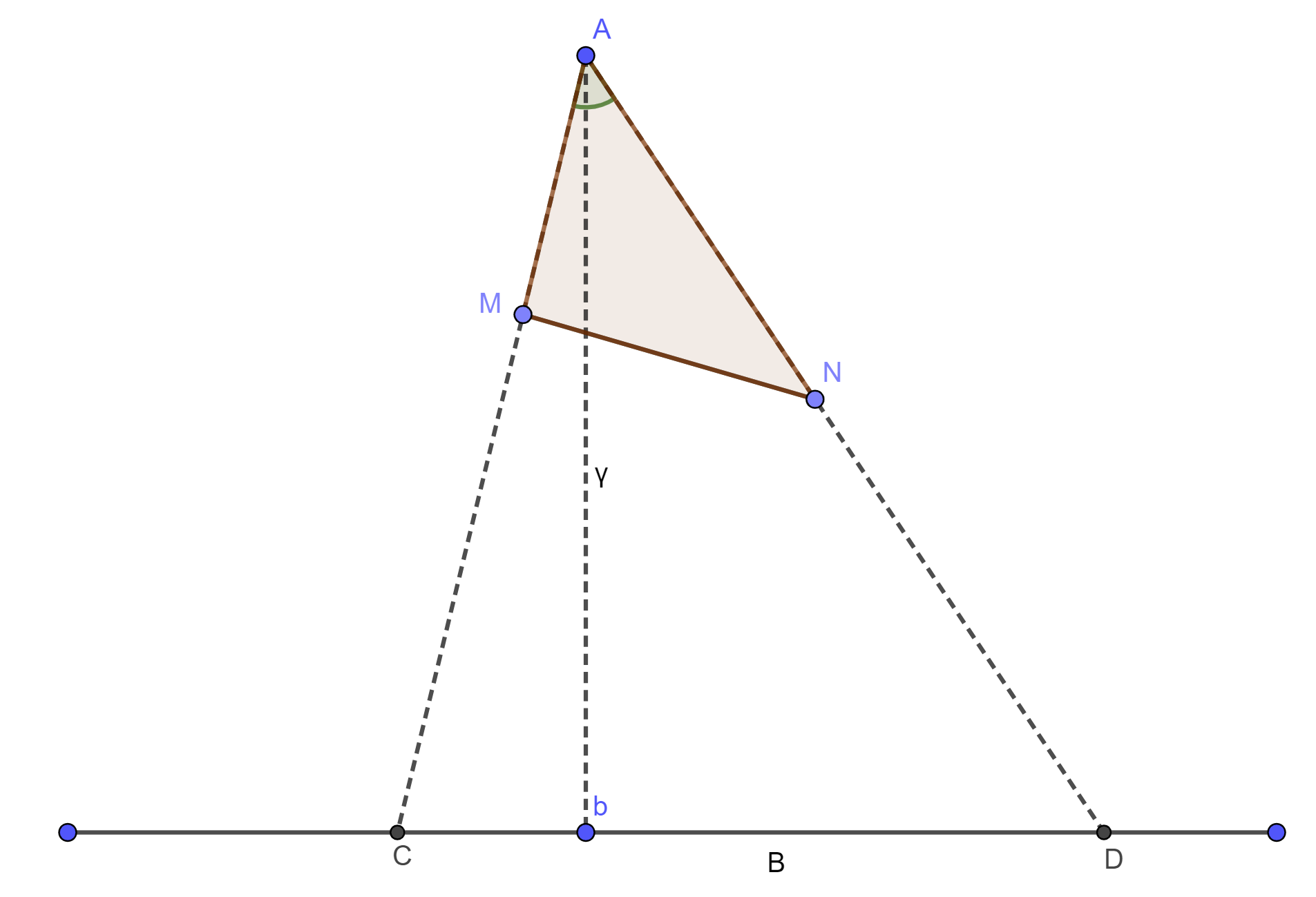}
\caption{Triangle $AMN$ and pencil $P(A,]\theta_{0},\theta_{1}[)$ of trajectories crossing boundary saddle connection $B$.} 
\end{figure}

\begin{proof}
Since trajectory $\gamma$ is compact there is a neighborhood of $\gamma$ that contains no singularity other than $A$. Thus $\gamma$ belongs to a pencil of trajectories starting from $A$ and crossing $B$ in a regular point. We denote by $P(A,]\theta_{0},\theta_{1}[)$ the maximal family of these trajectories. Trajectories starting from $A$ and following directions $\theta_{0},\theta_{1}$ do not hit a regular point of $B$ (otherwise the family could be extended further).\newline
Trajectories of $P(A,]\theta_{0},\theta_{1}[)$ do not intersect each other (there is no digon a dilation surface), do not hit any singularity and thus coincide with the affine embedding of the interior of a flat triangle $\mathcal{T}$ (the directions of the sides of the triangle are $\theta_{0}$, $\theta_{1}$ and the direction of $B$).\newline
Since $X$ contains no hyperbolic cylinder of angle at least $\pi$, Theorem \ref{thm:Veech} implies that the affine immersion of any disk extends continuously to its boundary. It is straightforward that the embedding of $\mathcal{T}$ thus continuously extends (as an immersion) to its boundary. The image of $\mathcal{T}$ is $ACD$ where $C,D \in B$ are the ends of the open interval $]C,D[$ formed by regular points of saddle connection $B$ crosses by some trajectory of pencil $P(A,]\theta_{0},\theta_{1}[)$. The image of the boundary is formed by geodesic segments.\newline
By hypothesis, trajectories starting from $A$ and following directions $\theta_{0},\theta_{1}$ do not cross $B$ at a regular point. Thus, both sides $[AC]$ and $[AD]$ contain a singularity. We denote by $M$ the singularity in side $[AC]$ that is the closest of $A$. We define $N$ similarly in $[AD]$ (it may happen that $M,N$ coincide with $C,D$ and are the ends of saddle connection $B$). Triangle $AMN$ satisfies all the required properties (see Figure~4).\newline
\end{proof}

\section{Exceptional pencils and dilation surfaces with boundary}\label{sec:EPD}

\begin{defn}
In a dilation surface $X$, a pencil $P(x_{0},U)$ with $x_{0} \in X$ and $U \subset \mathbb{S}^{1}$ is \textit{exceptional} if none of its trajectories cross the boundary or accumulate on a hyperbolic closed geodesic.
\end{defn}

The key result in this section is that exceptional pencils do not exist in dilation surfaces with boundary.

\begin{prop}\label{prop:noEP}
In any dilation surface with nonempty boundary, there can be no exceptional pencil.
\end{prop}

In Subsection~\ref{sec:Delta}, we introduce a specific class $(\Delta)$ of dilation surfaces with boundary to prove Proposition~\ref{prop:noEP}. In Subsection~\ref{sec:Existence} we deduce from Proposition~\ref{prop:noEP} the existence of hyperbolic cylinders in dilation surfaces with horizon saddle connections.

\subsection{Class $(\Delta)$ and exceptional pencils}\label{sec:Delta}

We introduce a class $(\Delta)$ of dilation surfaces with boundary such that any counterexample to Proposition \ref{prop:noEP} would lead to the existence of a counter-example inside class $(\Delta)$. We will then prove that class $(\Delta)$ is empty and infer that no dilation surface with nonempty boundary contains an exceptional pencil.

\begin{defn}
A dilation surface $X$ with nonempty boundary belongs to class $(\Delta)$ if it satisfies the following properties:
\begin{itemize}
\item $X$ contains an exceptional pencil $P(x_{0},U)$ with $x_{0} \in X$ and $U \subset \mathbb{S}^{1}$;
\item $x_{0}$ is a singularity of $X$;
\item there is no hyperbolic closed geodesic in any direction of $U \cup -U$;
\item $X$ has a geodesic triangulation (see Theorem \ref{thm:Veech}).
\end{itemize}
\end{defn}

\begin{prop}\label{prop:Delta}
If a dilation surface $X$ with nonempty boundary has an exceptional pencil, then class $(\Delta)$ is nonempty.
\end{prop}

\begin{proof}
Let $P(x_{0},U)$ be an exceptional pencil in $X$. We assume that there is a hyperbolic cylinder $C$ and an open interval $V \subset U$ such that the directions of closed geodesics of $C$ coincide with $V$ or $-V$ ($C$ may be only a portion of a bigger cylinder covering a bigger interval of directions). Since $P(x_{0},U)$ is an exceptional pencil, $P(x_{0},V)$ is also an exceptional pencil.\newline
If $x_{0} \in C$, then $x_{0}$ belongs to some hyperbolic closed geodesic $\alpha$ of $C$. The direction of $\alpha$ belongs to $V$ or $-V$ so $\alpha$ coincides with some trajectory in pencil $P(x_{0},V)$. This implies that $P(x_{0},V)$ is not an exceptional pencil.\newline
If $x_{0} \notin C$ and a trajectory $T$ of $P(x_{0},V)$ enters $C$, then Lemma \ref{lem:hypcyl} implies that $T$ accumulates on some hyperbolic closed geodesic of $C$. This would also contradict the assumption that $P(x_{0},V)$ is an exceptional pencil. Consequently, in any case, $P(x_{0},V)$ is disjoint from cylinder $C$.\newline
We can remove $C$ from $X$ and obtain a new dilation surface that still contains an exceptional pencil. If this leads to the existence of boundary components without any singular points, we can simply mark points on the boundary. Since there are at most finitely many hyperbolic closed geodesics in any given direction (see Lemma~\ref{lem:finitegeo}), after finitely many modifications, we obtain a dilation surface $X_{1}$ with nonempty boundary and an exceptional pencil $P(x_{0},W)$ such that no hyperbolic closed geodesic of $X_{1}$ belongs to a direction of $W$ or $-W$.\newline
Dilation surface $X_{1}$ could contain a hyperbolic cylinder of angle greater than or equal to $\pi$ because such a cylinder would contain a closed geodesic in a direction of $W$ or $-W$. Applying Veech's criterion (Theorem~\ref{thm:Veech}) then implies that $X_{1}$ decomposes into finitely many flat triangles whose sides are saddle connections. Lastly, if $x_{0}$ is not already a singularity of $X_{1}$, we can still make it a marked point. Therefore $X_{1}$ belongs to class $(\Delta)$.\newline
\end{proof}

For a dilation surface of class $(\Delta)$, the number of triangles of any geodesic triangulation is a topological invariant. We define $(\Delta_{min})$ as the subclass of $(\Delta)$ formed by dilation surfaces with the smallest possible number of triangles. We prove that surfaces of $(\Delta_{min})$ have very constrained dynamical properties.

\begin{lem}\label{lem:DeltaMin}
Let us consider a dilation surface $X$ of $(\Delta_{min})$ with a singularity $x_{0} \in X$, an exceptional pencil $P(x_{0},U)$ and no hyperbolic closed geodesic whose directions belongs to $U \cup -U$. We also require that no direction of a boundary saddle connection of $\partial X$ belongs to $U \cup -U$.

Then for any boundary saddle connection $B$, there is a boundary saddle connection $B'$ such that for any regular point $b \in B$, every trajectory of pencil $P(b,U)$ or $P(b,-U)$ (one of them is nonempty) crosses $B'$ at some regular point.
\end{lem}

\begin{proof}
We consider a regular point $b$ in boundary saddle connection $B$ such that $P(b,\epsilon U)$ is nonempty for some $\epsilon = \pm 1$. We start by proving that no trajectory in pencil $P(b,\epsilon U)$ can hit a singularity.\newline
Otherwise such a trajectory $\gamma$ hits a singularity $A$, and Lemma~\ref{lem:triangle} proves the existence of a flat triangle $AMN$ (where $M,N$ are singularities of $X$) such that:
\begin{itemize}
\item denoting by $]\theta_{0},\theta_{1}[$ the open angular sector of vertex $A$ inside $AMN$, trajectory $\gamma$ (with the opposite orientation) belongs to pencil $P(A,]\theta_{0},\theta_{1}[)$;
\item every trajectory in pencil $P(A,]\theta_{0},\theta_{1}[)$ crosses boundary saddle connection $B$ at some regular point.    
\end{itemize}
Now we consider an open subset $V \subset U$ such that $-\epsilon V \subset ]\theta_{0},\theta_{1}[$. Since $P(x_{0},U)$ is an exceptional pencil, $P(x_{0},V)$ is also an exceptional pencil. We are going to prove that $P(x_{0},V)$ is disjoint from triangle $AMN$.\newline
We start by considering the case where $\epsilon=-1$. Every trajectory traveling in a direction of $]\theta_{0},\theta_{1}[$ (including $V$ in this case) that crosses triangle $AMN$ eventually crosses boundary saddle connection $B$ (see Figure~4). Consequently, if $\epsilon = -1$, pencil $P(x_{0},V)$ is disjoint from triangle $AMN$.\newline
If $\epsilon=1$, assuming that some trajectory in pencil $P(x_{0},V)$ crosses triangle $AMN$, its starting point $x_{0}$ belongs to pencil $P(A,]\theta_{0},\theta_{1}[)$ (see Figure~4). However, we know by hypothesis that no trajectory in pencil $P(A,]\theta_{0},\theta_{1}[)$ can hit a singularity. Thus $P(x_{0},V)$ is also disjoint from triangle $AMN$ in that case.\newline
Now since $P(x_{0},V)$ is disjoint from triangle $AMN$, we can remove $AMN$ from surface $X$ and obtain a surface $X'$ in class $(\Delta)$ that contradicts the fact that $X$ belongs to $(\Delta_{min})$. Indeed, pencil $P(x_{0},V)$ is still an exceptional pencil and $X'$ does not contain any hyperbolic closed geodesic in any direction in $V \cap -V$. Thus, no trajectory in pencil $P(b,\epsilon U)$ can hit a singularity.\newline

Unless some trajectory in pencil $P(b,\epsilon U)$ crosses a boundary saddle connection, Lemma~\ref{lem:immercone} proves the existence of a hyperbolic closed geodesic in some direction of $U \cup -U$. Starting from a trajectory of $P(b,\epsilon . U)$ crossing some boundary saddle connection $B'$, a continuity argument involving Theorem \ref{thm:Veech} and similar to the proof of Lemma \ref{lem:triangle} then proves that every trajectory of the pencil crosses $B'$ (otherwise one trajectory would hit a singularity). A similar statement holds for any starting point in $B$. Since these pencils intersect each other, it is straightforward that the trajectories of these pencils eventually cross the same boundary saddle connection $B'$.
\end{proof}

Now we are finally able to rule out the existence of exceptional pencils in any dilation surface with nonempty boundary.

\begin{proof}[Proof of Proposition \ref{prop:noEP}]
If there exists a dilation surface with nonempty boundary and an exceptional pencil, then Proposition~\ref{prop:Delta} proves that class $(\Delta)$ is nonempty. We will consider a surface of subclass $(\Delta_{min})$ and obtain a contradiction.\newline
Set $X$ a dilation surface in $(\Delta_{min})$ with an exceptional pencil $P(x_{0},U)$ and no hyperbolic closed geodesic that has its direction in $U \cup -U$. Up to taking an open subset of $U$, we can assume that no direction of any boundary saddle connection in $X$ belongs to $U \cup -U$. Then Lemma~\ref{lem:DeltaMin} proves that for every regular point $b$ of a boundary saddle connection $B$, trajectories of the nonempty pencils $P(b,\epsilon . U)$ (for $\epsilon = \pm 1$) cross the same boundary saddle connection $B'$ in a regular point.\newline
Looking at these trajectories with the opposite orientation, we use Lemma~\ref{lem:DeltaMin} once again to prove that for every regular point $b'$ chosen from boundary saddle connection $B'$, trajectories of the nonempty pencils $P(b',-\epsilon . U)$ cross saddle connection $B$. This automatically implies that one trajectory of the pencils hits one endpoint of the opposite boundary saddle connection.\newline
\end{proof}

\subsection{Existence of hyperbolic cylinders}\label{sec:Existence}

The fact that exceptional pencils do not exist directly implies the existence of hyperbolic closed geodesics.

\begin{cor}\label{cor:hyperExist}
In any dilation surface $X$ with nonempty boundary, if there is a pencil $P(x_{0},U)$ such that no trajectory of the pencil crosses the boundary, then there is an open subset $V \subset U$ such that every trajectory of $P(x_{0},V)$ accumulates on some hyperbolic closed geodesic.
\end{cor}

\begin{proof}
Such a pencil $P(x_{0},U)$ is not exceptional because this would contradict Proposition~\ref{prop:noEP}. Thus, at least one of its trajectories $\gamma$ accumulates on some hyperbolic closed geodesic. These geodesics belong to a hyperbolic cylinder. Even if the hyperbolic closed geodesic is singular, we can perturb it to one side or the other so that the perturbed trajectory $\gamma'$
accumulates on a closed geodesic in the same cylinder.
\end{proof}

A consequence of Corollary~\ref{cor:hyperExist} is that dilation surfaces with horizon saddle connections satisfy Conjecture~\ref{conj:two}.

\begin{cor}\label{cor:horconj2}
In any (closed) dilation surface $X$ with at least one horizon saddle connection $L$, directions of hyperbolic closed geodesics form a dense subset of $\mathbb{S}^{1}$. 
\end{cor}

\begin{proof}
Lemma \ref{lem:horipencil} implies that in any open set of directions, we can find a pencil of trajectories in $X \setminus L$ such that none of these trajectories cross the boundary. Corollary~\ref{cor:hyperExist} then implies that some trajectories of this pencil accumulate on hyperbolic closed geodesics. Therefore, there are directions of hyperbolic closed geodesics in any open subset of $\mathbb{S}^{1}$.\newline
\end{proof}

\section{Morse-Smale dynamics}\label{sec:MSD}

Using results on pencils that we highlighted in Section~\ref{sec:EPD}, we prove that the existence of hyperbolic closed geodesics implies that the directional flow satisfies generically the Morse-Smale property.

\begin{prop}\label{prop:MS}
In any closed dilation surface $X$, if there is a hyperbolic cylinder $C$ such that the directions of its closed geodesics cover an interval $I \subset \mathbb{S}^{1}$, then there is an open dense symmetric subset of $I \cup -I$ such that any trajectory of $X$ in a direction belonging to $J$ either accumulates on a hyperbolic closed geodesics or hits a singularity.
\end{prop}

\begin{proof}
We consider any open interval $U \subset I$ of angle strictly smaller than $\pi$. Let $C'$ be a portion of cylinder $C$ covering all directions in $U$. Adding marked points on the two boundary loops, we therefore endow $X \setminus C'$ with a structure of dilation surface with boundary. For any $x \in X \setminus C'$, no trajectory in pencils $P(x,-I)$ enters $C'$. There are finitely many such pencils of separatrices (trajectories starting from singularities). For each of them, Corollary \ref{cor:hyperExist} implies that there is an open dense subset of $-U$ where trajectories accumulate on some hyperbolic closed geodesic. Their intersection $M$ is an open dense subset of $-U$ such that every separatrix of $X \setminus C'$ in these directions accumulates on a (nonsingular) hyperbolic closed geodesic.\newline
Let $d$ be a direction in $M$. Cylinders containing hyperbolic closed geodesics in direction
$d$ are disjoint. We denote by $C(d)$ the union of the portions of these cylinders covering directions of $-U$. We also denote by $N(d)$ the intersection of the different intervals of directions of the (oriented) closed geodesics of these portions. $N(d)$ is an open interval containing $d$.\newline
Set $X \setminus C(d)$ can be endowed with a structure of dilation surface with boundary. Separatrices of $X \setminus C(d)$ in directions of $-N(d)$ cannot enter any cylinder of $C(d)$. Therefore, we can use Corollary~\ref{cor:hyperExist} once again to prove that there is an open dense subset $R(d)$ of $-N(d)$ such that every separatrix of $X \setminus C(d)$ in a direction of $R(d)$ accumulates on a (nonsingular) hyperbolic closed geodesic.\newline
This way, we prove that there is an open dense subset $S$ of $I$ in the directions of which every separatrix accumulates on a hyperbolic closed geodesic. Similarly, there is an open dense subset of $-I$ for which the same property holds. Since $S \cap -T$ is an open dense subset of $I$ and $T \cap -S$ is an open dense subset of $-I$, it follows that $J = (S \cap -T) \cup (T \cup -S)$ is a symmetric open dense subset of $I \cup -I$ in the directions of which every separatrix accumulates on a hyperbolic closed geodesic.\newline
For any direction $r \in J$, closed geodesics in directions $r$ and $-r$ belong to a disjoint union $C(r;-r)$ of portions of cylinders. Since there are no saddle connection in these directions, cylinders of $C(r;-r)$ automatically are hyperbolic.
By removing these portions of cylinders from $X$ and adding the adequate marked points, we create a dilation surface with boundary where every separatrix eventually crosses the boundary (and then accumulates on a hyperbolic closed
geodesic inside, see Lemma~\ref{lem:hypcyl}).\newline
Then Proposition~\ref{prop:separatrix} implies that in $X \setminus C(r;-r)$, any trajectory traveling in directions $r$ or $-r$ either hits a singularity or crosses the boundary. Consequently, every trajectory in $X$ pointing in these directions either hits a singularity or accumulates on some hyperbolic closed geodesic.
\end{proof}

We recall that a flow is Morse-Smale if every trajectory is either critical (hits a singularity) or accumulates on some attracting limit cycle.

\begin{proof}[Proof of Theorem \ref{thm:equi}]
If a dilation surface satisfies Conjecture~\ref{conj:two}, then the directions of its hyperbolic closed geodesics are dense in $\mathbb{S}^{1}$. Each of these geodesics belongs to a hyperbolic cylinder for which directions of regular closed geodesics cover an open subset of $\mathbb{S}^{1}$. For each of these open sets, Proposition~\ref{prop:MS} provides an open dense subset on which the dynamics of the directional flow is Morse-Smale. The union of these open sets forms an open dense subset of $\mathbb{S}^{1}$.
\end{proof}

\begin{proof}[Proof of Theorem \ref{thm:horizon}]
Corollary~\ref{cor:horconj2} proves that every closed dilation surface with a horizon saddle connection satisfies Conjecture~\ref{conj:two}. Theorem~\ref{thm:equi} proves that Conjecture~\ref{conj:two} implies Conjecture~\ref{conj:three}.\newline
\end{proof}

\paragraph{\bf Acknowledgements.} The author thanks Adrien Boulanger, Selim Ghazouani, Ben-Michael Kohli and the anonymous referees for valuable remarks and inspiration.\newline

\nopagebreak
\vskip.5cm
\end{document}